\documentclass[12pt]{amsart}
\usepackage{latexsym}
\usepackage{amsthm}
\usepackage{amsmath}
\usepackage{amsfonts}
\usepackage{amssymb}
\usepackage[all]{xy}

\theoremstyle{plain}
\newtheorem{theo}{Theorem}[section]
\newtheorem{lemma}[theo]{Lemma}
\newtheorem{propo}[theo]{Proposition}
\newtheorem{coro}[theo]{Corollary}

\theoremstyle{definition}
\newtheorem{defi}[theo]{Definition}
\newtheorem{rem}[theo]{Remark}
\newtheorem{exam}[theo]{Example}

\newtheorem{notation}[theo]{Notation}

\newcommand\Po{\operatorname{Po}}
\newcommand\Tc{\operatorname{Tc}}
\newcommand\Tcsh[1]{\operatorname{#1\textrm{-}Tc}}
\newcommand\Rt{\operatorname{Rt}}
\newcommand\Gd{\operatorname{Gd}}
\newcommand\Gdcard[1]{\operatorname{#1\textrm{-}Gd}}

\newcommand\GdDircard[1]{\operatorname{#1\textrm{-}GdDir}}

\newcommand\Mono{\operatorname{Mono}}
\newcommand\Epi{\operatorname{Epi}}

\newcommand\id{\operatorname{id}}

\newcommand\Set{\operatorname{\bf Set}}

\newcommand\SSet{\operatorname{\bf SSet}}

\newcommand\cof{\operatorname{cof}}
\newcommand\cell{\operatorname{cell}}
\newcommand\colim{\operatorname{colim}}

\newcommand\ca{\mathcal {A}}
\newcommand\cb{\mathcal {B}}
\newcommand\cc{\mathcal {C}}

\newcommand\ci{\mathcal {I}}

\newcommand\ck{\mathcal {K}}
\newcommand\cl{\mathcal {L}}

\newcommand\crr{\mathcal {R}}
\newcommand\cs{\mathcal {S}}

\newcommand\cp{\mathcal {P}}

\newcommand\cx{\mathcal {X}}

\newcommand\bz{\mathbb {Z}}

\newlength{\hlp}
\newcommand{\leftbox}[2]{\settowidth{\hlp}{$#1$}\makebox[\hlp][l]{${#1}{#2}$}}

\newcommand{\stdpbsize}{15pt}
\newcommand{\stdpboffset}{.5}

\newcommand{\xycorner}[3]{\save #2="a";#1;"a"**{}?(\stdpboffset);"a"**\dir{-};#3;"a"**{}?(\stdpboffset);"a"**\dir{-} \restore}
\newcommand{\po}{\xycorner{[]+<-\stdpbsize,0pt>}{[]+<-\stdpbsize,-\stdpbsize>}{[]+<0pt,-\stdpbsize>}}
 
\date{April 25, 2013}
 
\begin{document}
\title[On a fat small object argument]
{On a fat small object argument}
\author[M. Makkai, J. Rosick\'{y} and L. Vok\v r\'\i{}nek]
{M. Makkai$^{*}$, J. Rosick\'{y}$^{**}$ and L. Vok\v r\'\i{}nek$^{**}$}
\thanks{$^{*}$ Supported by the project CZ.1.07/2.3.00/20.0003 of the Operational Programme Education for Competitiveness of the Ministry of Education, Youth and Sports of the Czech Republic.\\\indent $^{**}$ Supported by the Grant Agency of the Czech republic under the grant P201/12/G028.} 
\address{
\newline 
Department of Mathematics and Statistics\newline
Masaryk University, Faculty of Sciences\newline
Kotl\'{a}\v{r}sk\'{a} 2, 611 37 Brno, Czech Republic\newline
makkai@math.mcgill.ca\newline
rosicky@math.muni.cz\newline
koren@math.muni.cz
}
 
\begin{abstract}
Good colimits introduced by J.~Lurie generalize transfinite composites and provide an important tool for understanding cofibrant generation in locally presentable categories. We will explore the relation of good colimits to transfinite composites further and show, in particular, how they eliminate the use of large objects in the usual small object argument.
\end{abstract} 

\maketitle
 
\section{Introduction}
Combinatorial model categories were introduced by J.~H.~Smith as model categories which are locally presentable and cofibrantly generated.
The latter means that both cofibrations and trivial cofibrations are cofibrantly generated by a set of morphisms. He has not published his results 
but most of them can be found in \cite{B}, \cite{D}, \cite{L} and \cite{R}. A typical feature of a combinatorial model category $\ck$ is the existence
of a regular cardinal $\lambda$ such that everything happens below $\lambda$, i.e., among $\lambda$-presentable objects, and then it is extended
to $\ck$ by using $\lambda$-filtered colimits. In particular, fibrant objects form a $\lambda$-accessible category and any cofibrant
object is a $\lambda$-filtered colimit of $\lambda$-presentable cofibrant objects. In general, this cardinal $\lambda$ is greater than $\kappa$
in which $\ck$ is presented. This means that $\ck$ is $\kappa$-combinatorial in the sense that it is locally $\kappa$-presentable and both 
cofibrations and trivial cofibrations are cofibrantly generated by a set of morphisms between $\kappa$-presentable objects. For example,
the model category $\SSet$ of simplicial sets is $\omega$-combinatorial but finitely presentable simplicial sets have $\omega_1$-presentable
fibrant replacements. One of our main results is that any cofibrant object in a $\kappa$-combinatorial model category is a $\kappa$-filtered colimit
of $\kappa$-presentable cofibrant objects. The proof is based on the concept of a good colimit.
Good colimits were introduced by Lurie in \cite{L} and studied by the first author in \cite{M}. They generalize transfinite composites
but, while transfinite composites are thin and include large objects, good colimits are fat but their objects can be made small. This leads
to the just mentioned result and may be called a \textit{fat small object argument}. One of its consequences is the result of Joyal and Wraith
\cite{JW} that any acyclic simplicial sets is a filtered colimit of finitely presentable acyclic simplicial sets. The original motivation for the introduction of good limits in \cite{L} was to prove that a retract of a cellular morphism is cellular in retracts (of small cellular morphisms).

It is remarkable that the same idea independently emerged in mo\-du\-le theory where the Hill lemma was used for the same purpose (see \cite{S}).
In particular, \cite{S} shows, in this additive setup, that a retract of a cellular morphism is cellular in retracts. Both the model category and the module theory situation subsumes into the framework of a locally presentable category equipped
with a cofibrantly generated weak factorization system. We are working in this context and show how filtered colimits mix with those used
in cofibrant generation. More results in this direction will be presented in \cite{MR}.

In the appendix, we reprove Lurie's result about the elimination of retracts using $\kappa$-good colimits which are moreover $\kappa$-directed (this is a major departure from Lurie's approach). Such colimits play a central role in our paper, and are essential for our applications.

\section{Weak factorization systems}
Let $\ck$ be a category and $f\colon A\to B$, $g\colon C\to D$ morphisms such that in each commutative square
$$
\xymatrix{
A \ar[r]^{u} \ar[d]_{f}& C \ar[d]^g\\
B\ar[r]_v & D
}
$$
there is a diagonal $d\colon B \to C$ with $df=u$ and $gd=v$. Then we say that $g$ has the \textit{right lifting property}
w.r.t.\ $f$ and $f$ has the \textit{left lifting property} w.r.t.\ $g$. For a class $\cx$ of morphisms of $\ck$ we put
\begin{align*}
\cx^{\square}& = \{g\mid g \ \mbox{has the right lifting property
w.r.t.\ each $f\in \cx$\} and}\\
{}^\square\cx & = \{ f\mid f \ \mbox{has the left lifting property
w.r.t.\ each $g\in \cx$\}.}
\end{align*}

\noindent
A \textit{weak factorization system} $(\cl,\crr)$ in a category $\ck$ consists of two classes $\cl$ and $\crr$ of morphisms of $\ck$ such that
\begin{enumerate}
\item[(1)] $\crr = \cl^{\square}$, $\cl = {}^\square \crr$, and
\item[(2)] any morphism $h$ of $\ck$ has a factorization $h=gf$ with
$f\in \cl$ and $g\in \crr$.
\end{enumerate}

\noindent
A weak factorization system $(\cl,\crr)$ is called \textit{cofibrantly generated} if there is a set $\cx$ of morphisms such that $\crr=\cx^\square$.

\begin{notation}\label{not2.1}
{
In order to state closure properties of the class $\cl$, we introduce the following notation. Let $\cx$ be a class of morphisms in $\ck$. 

(1) $\Po(\cx)$ denotes the class of pushouts of morphisms in $\cx$: $f\in\Po(\cx)$ iff $f$ is an isomorphism or there is a pushout diagram
$$
\xymatrix{
A \ar[r]^{f} & B \po \\
X \ar [u]^{} \ar [r]_{g} &
Y \ar[u]_{}
}
$$
 
with $g\in\cx$.

(2) $\Tc(\cx)$ denotes the class of transfinite composites (= compositions) of morphisms from $\cx$: $f\in\Tc(\cx)$ iff there is a smooth chain 
$(f_{ij}\colon A_i \to A_j)_{i\leq j\leq \lambda}$ (i.e., $\lambda$ is an ordinal, $(f_{ij}\colon A_i \to A_j)_{i<j}$ is a colimit for any limit ordinal 
$j\leq\lambda$) such that $f_{i,i+1}\in\cx$ for each $i< \lambda$ and $f=f_{0\lambda}$.
 
(3) $\Rt(\cx)$ denotes the class of retracts of morphisms in $\cx$ in the category $\ck^2$ of morphisms of $\ck$.

(4) $\cell(\cx)=\Tc\Po(\cx)$ denotes the \textit{cellular closure} of $\cx$; the elements of $\cell(\cx)$ are called \emph{$\cx$-cellular maps} or \emph{relative $\cx$-cell complexes}, and

(5) $\cof(\cx)=\Rt\Tc\Po(\cx)$ the \textit{cofibrant closure} of $\cx$; the elements of $\cof(\cx)$ are called \emph{$\cx$-cofibrations} or simply \emph{cofibrations}.

A basic property of a locally presentable category $\ck$ is that the pair $(\cof(\cx),\cx^\square)$ is a weak factorization system for any set $\cx$ of morphisms. In particular
$$
{}^\square(\cx^\square)=\cof(\cx)
$$ 
(``small object argument''); see \cite{B}.

Later, we will use the following simple observation: in the above (defining) equality $\cof(\cx)=\Rt\Tc\Po(\cx)$, it is sufficient to consider retractions whose domain components are the identity morphisms, i.e.\ retractions taking place in the respective under category $A/\ck$. This is beacause any retract $f$ of $g$ can be expressed also as a retract of the pushout $g'$ of $g$ along the retraction $r$ of the domains,
\[\xymatrix{
B \ar[r] & Y \ar[r] \ar@/^10pt/[rr] & Y' \ar@{-->}[r] \po & B \\
A \ar[r] \ar[u]_-f & X \ar[r]_-r \ar[u]_-g & A \ar[u]_-{g'} \ar@/_5pt/[ru]_-f
}\]

(6) Let $\ck_\kappa$ denote the full subcategory of $\ck$ consisting of $\kappa$-presentable objects and $(\ck_\kappa)^2$ the category of morphisms of $\ck_\kappa$. 
For a class $\cx\subseteq(\ck_\kappa)^2$, we put
$$
\Po_\kappa(\cx)=\Po(\cx)\cap(\ck_\kappa)^2
$$
$$
\cell_\kappa(\cx)=\cell(\cx)\cap(\ck_\kappa)^2
$$
$$
\cof_\kappa(\cx)=\cof(\cx)\cap(\ck_\kappa)^2
$$
and denote $\Tcsh\kappa(\cx)$ the class of transfinite composites of length smaller than $\kappa$, i.e., $\lambda<\kappa$ in (2).  
}
\end{notation}
 
\section{Finite fat small object argument}\label{s:finite_fsoa}
In this short section, we outline our fat small object argument in the case $\kappa=\aleph_0$. We assume that $\ck$ is locally finitely presentable and that the set $\cx$ consists of morphisms between finitely presentable objects. For simplicity, we will also assume that all $\cx$-cofibrations are regular monomorphisms (this assumption can be removed).

Let $f\colon A\to B$ be a morphism. A presented finite cell complex (see \cite[Section 10.6]{Hi}) in $A/\ck/B$ (i.e.\ the category of objects of $\ck$ under $A$ and over $B$) is a finite sequence $C$ of the form
\[A=A_0\to A_1\to\cdots\to A_n\to B,\]
whose composition is $f$, together with an expression of each $A_{i-1}\to A_i$ as a pushout of a finite coproduct of elements of $\cx$,
\[\xymatrix{
A_{i-1} \ar[r] & A_i \po \\
\bigsqcup\limits_{j\in J_i}X_j \ar[r]_{\bigsqcup g_j} \ar[u] & \bigsqcup\limits_{j\in J_i}Y_j \ar[u]
}\]
We call the components $X_j\to A_{i-1}$ of the left vertical map the characteristic maps and assume that they are all different (no two cells are glued along the same map at the same step). We stress that the sets $J_i$ and the characteristic maps are taken as a part of the structure of a presented cell complex. We denote $|C|=A_n$ the ``total space'' of $C$, it is naturally an object of $A/\ck/B$. We also define $A_m=A_n$ for $m\geq n$. In this way a presented finite cell complex can be prolonged arbitrarily.

A subcomplex inclusion $\iota\colon C\to C'$ is a sequence of morphisms $\iota_i\colon A_i\to A_i'$ for which there exists a diagram
\[\xymatrix{
A_{i-1} \ar[d]_-{\iota_{i-1}} & \bigsqcup\limits_{j\in J_i}X_j \ar[l] \ar[r] \ar[d] & \bigsqcup\limits_{j\in J_i}Y_j \ar[d] \\
A_{i-1}' & \bigsqcup\limits_{j\in J_i'}X_j \ar[l] \ar[r]& \bigsqcup\limits_{j\in J_i'}Y_j
}\]
with the two unlabelled vertical maps the inclusions of sub-coproducts, corresponding to $J_i\subseteq J_i'$, and the induced map on pushouts being $\iota_i$. We denote by $|\iota|=\iota_n$ the top part of the sequence $\iota$, where $n$ is at least the length of $C$ and $C'$.

The presented finite cell complexes together with their inclusions form a directed poset (this uses the assumption on cofibrations being regular monomorphisms --- otherwise, it would not have been even a poset). For details, see \cite[Section~10.6 and Chapter~12]{Hi}. The total space functor thus provides a directed diagram in $A/\ck/B$. We form its colimit --- a factorization
\[A\to A'\to B.\]
of the map $f$. The first map can be seen to lie in $\cell(\cx)$ (e.g.\ as a consequence of Proposition~\ref{prop4.5}). Now, we will show that the second map lies in $\cx^\square$. Given a commutative square
\[\xymatrix{
X \ar[r] \ar[d] & A' \ar[d] \\
Y \ar[r] & B
}\]
with the left map in $\cx$, we use finite presentability of $X$ to factor the morphism $X\to A'=\colim_{C}|C|$ through some finite cell complex $|C|$ in the diagram. By prolonging the presentation of $C$ by one step, we construct a new presented cell complex $C'$ with $|C'|=|C|\sqcup_XY$. The composition $Y\to|C'|\to A'$ is then a diagonal in the above square.

We call this the fat small object argument, as it does not express $A'$ as a ``long'' transfinite composite, but rather as a colimit of a spread out diagram where all objects are obtained using a finite number of cells. Thus, when $A$ and all domains and codomains of $\cx$ are finitely presentable, the same applies to all objects in this diagram and $A'$ is a directed colimit of a digram of finitely presentable objects.

In the proceeding, we will characterize the ``good'' diagrams arising in the fat small object argument and show that their ``composites'' lie in $\cell(\cx)$ in general. When one comes to uncountable $\kappa$, much more care has to be taken. The diagram obtained from cell complexes with $<\kappa$ cells is not good anymore (it is still $\kappa$-directed and has the desired colimit). We will describe this general case in more detail in an appendix.

\section{Good colimits}
Recall that a poset $P$ is well-founded if every of its nonempty subsets contains a minimal element. Given $x\in P$, $\downarrow x=\{y\in P\mid y\leq x\}$ 
denotes the initial segment generated by $x$.

\begin{defi}\label{def4.1} 
We say that a poset $P$ is \textit{good} if it is well-founded and has a least element $\perp$. A good poset
is called $\kappa$-\textit{good} if all its initial segments $\downarrow x$ have cardinality $<\kappa$.
\end{defi}

Any well-ordered set is good and every finite poset with a least ele\-ment is good; in particular, the shape poset for pushout, a three-element good poset which is not a chain. The following terminology is transferred from well-ordered sets.
 
An element $x$ of a good poset $P$ is called \textit{isolated} if 
$$
\downdownarrows x=\{y\in P\mid y< x\}
$$ 
has a top element $x^-$ which is called the \textit{predecessor} of $x$. A non-isolated element distinct from $\perp$ is called \textit{limit}.
Given $x<y$ in a poset $P$, we denote $xy$ the unique morphism $x\to y$ in the category $P$.

\begin{defi}\label{def4.2} 
A diagram $D\colon P\to\ck$ is \emph{smooth} if, for every limit $x\in P$, the diagram $(D(yx)\colon Dy\to Dx)_{y<x}$ is a colimit cocone on the restriction of $D$ to $\downdownarrows x$.

A \textit{good} diagram $D\colon P\to\ck$ is a smooth diagram whose shape category $P$ is a good poset.
\end{defi}

\begin{exam}\label{ex4.3}
The canonical diagram from Section~\ref{s:finite_fsoa} is $\omega$-good.
\end{exam}

The \textit{links} in a good diagram $D\colon P\to\ck$ are the morphisms $D(x^-x)$ for the isolated elements $x\in P$. 

The following result can be found both in \cite{M} and in \cite{L}, A.1.5.6. The proof is ``the same" as for transfinite composites.

\begin{propo}\label{prop4.4} Let $(\cl,\crr)$ be a weak factorization system in a ca\-te\-go\-ry $\ck$ and $D\colon P\to\ck$ a good diagram with links 
in $\cl$. Then all components of a colimit cocone $\delta_x\colon Dx\to\colim D$ belong to $\cl$.
\end{propo}
\begin{proof}
Since the principal filter $\uparrow x$ is good for each $x\in P$, it suffices to show that the component $\delta_\perp\colon D\perp\to\colim D$ belongs 
to $\cl$. Consider
$$
\xymatrix{
D\perp \ar[r]^-{u} \ar[d]_-{\delta_\perp}& X \ar[d]^-g\\
\colim D\ar[r]_-{v} & Y
}
$$
with $g\in\crr$. It suffices to construct a compatible cocone $d_x\colon Dx\to X$ such that $d_\perp=u$ and $gd_x=v\delta_x$ for each $x\in P$.
Then the induced morphism $d\colon\colim D\to X$ is the desired diagonal in the square above. Since $P$ is well-founded, we can proceed by recursion.
Assume that we have $d_y$ for each $y<x$. If $x$ is limit we get $d_x$ as induced by the cocone $(d_y)_{y<x}$. If $x$ is isolated we get $d_x$ 
as the diagonal in the square
\[\xy*!D!<0pt,10pt>\xybox{\xymatrix{
Dx^- \ar[r]^-{d_{x^-}} \ar[d]_-{D(x^-x)}& X \ar[d]^-g\\
Dx\ar[r]_-v & Y
}}\endxy\qedhere\]
\end{proof}

The \textit{composite} of a good diagram $D\colon P\to\ck$ is the component $\delta_\perp$ of a colimit cocone. A \textit{good composite} of morphisms 
from $\cx$ is the composite of a good diagram with links in $\cx$. The just proved proposition says that $\cl$ is closed under good composites. 
This proposition can be strengthened as follows.
 
\begin{propo}\label{prop4.5}
Let $\cx$ be a class of morphisms in a cocomplete category $\ck$. Then the composite of a good diagram in $\ck$ with links in $\Po(\cx)$
belongs to $\cell(\cx)$.
\end{propo}
\begin{proof}
Let $D\colon P\to\ck$ be a good diagram with links in $\Po(\cx)$. Let $\preceq$ be a well-ordering of $P$ extending its partial ordering $\leq$
(see \cite{J}, 1.2, Theorem 5). Let $Q$ consist of all non-empty initial segments of $(P,\preceq)$. Then $(Q,\subseteq)$ is a well-ordered set. Consider the diagram $E\colon Q\to\ck$ such that $ES$ is a colimit of the restriction of $D$ to $S$ 
and $E_{SS'}\colon ES\to ES'$ is the induced morphism. It is easy to see that $E$ is a smooth transfinite sequence and by definition, $\colim D\cong EP$.

It remains to show that links of $E$ belong to $\Po(\cx)$. These links 
are precisely $E[x)\to E[x]$, where $[z]=\{y\in P\mid y\preceq z\}$ and $[z)=\{y\in P\mid y\prec z\}$. Treating both as subposets of $P$, we have a pushout diagram of categories and an induced pushout diagram of colimits in $\ck$:
$$
\xymatrix{
[x) \ar[r]^{} & [x] \po & E[x) \ar[r]^{E([x)[x])} & E[x] \po \\
\downdownarrows x \ar [u]^{} \ar [r]_{} & \downarrow x \ar[u]_{} & E(\downdownarrows x) \ar [u]^{} \ar [r] & E(\downarrow x) \ar[u]_{}
}
$$
When $x$ is isolated, the bottom map is $D(x^-x)$ and when $x$ is limit, it is an isomorphism. In both cases $E([x)[x])$ belongs to $\Po(\cx)$.
\end{proof}

The class of good composites of diagrams with links in $\cx$ will be denoted $\Gd(\cx)$. Analogously, $\Gdcard\kappa(\cx)$ denotes $\kappa$-good
composites with links in $\cx$.  A transfinite composite is $\kappa$-good if and only if it is of length $\leq\kappa$.

\begin{rem}\label{re4.6}
The proposition above can be refined as follows.

Let $\lambda$ be an infinite cardinal. Then the composite of a $\kappa$-good diagram of cardinality $<\lambda$ with links in $\Po(\cx)$ belongs to $\Tcsh\lambda\Po(\cx)$.

Moreover, if $\kappa$ is regular and $\lambda\geq\kappa$, then this composite can be expressed as a transfinite composite of length exactly $\lambda$ with links in $\Po(\cx)$. This follows from the fact that the well-ordering of $P$ from the proof of \ref{prop4.5} can be chosen isomorphic to the ordinal $\lambda$ (see \cite{J}, Theorem~5).
\end{rem}

\begin{notation}\label{not4.7}
{
Let $D\colon P\to\ck$ be a good diagram and $Q$ an initial segment of $P$. Then $\colim_Q D$ will denote the colimit of the restriction of $D$ on $Q$.
}
\end{notation}

\begin{rem}\label{re4.8}
As with most of our statements, there is also a \emph{relative} version: given an initial segment $Q\subseteq P$ and a diagram $D\colon P\to\ck$ 
such that the links $D(x^-x)$ lie in $\Po(\cx)$ for all $x\in P\smallsetminus Q$, the induced map on colimits $\colim_QD\to\colim D$ belongs 
to $\cell(\cx)$. The proof is the same, only with all the elements of $Q$ ignored. A particularly simple case of this relative version is the following 
lemma.
\end{rem}

\begin{lemma}\label{le4.9}
Let $D\colon P\to\ck$ be a good diagram and let $Q\subseteq P$ be an initial segment. If all the elements in $P\smallsetminus Q$ are limit, 
the induced map $\colim_QD\to\colim D$ is an isomorphism.\qed
\end{lemma}

Transfinite composites are thin and long and are used for a weak factorization of a morphism $h$. This procedure is called a ``small object 
argument". We will show how to convert a transfinite composite into a fat and short good composite. Our procedure can be called
a \textit{fat small object argument}.  

First we will prove an auxiliary lemma.

\begin{lemma}\label{le4.10}
Let $D\colon P\to\ck$ be a $\kappa$-good diagram. Then there exists its extension $D^*\colon P^*\to\ck$ to a $\kappa$-good $\kappa$-directed diagram. 
In this extension, $P\subseteq P^*$ is an initial segment and all the elements in $P^*\smallsetminus P$ are limit. In particular, the links of $D^*$ 
are exactly those of $D$ and the natural map $\colim D\to\colim D^*$ is an isomorphism.
\end{lemma}

\begin{proof}
We will construct $P^*$ and $D^*$ by iterating transfinitely the following construction. Let $P^+$ consists of adding, for each initial segment 
$S\subseteq P$ of cardinality $<\kappa$ without a greatest element, an element $p_S$ such that $s<p_S$ for each $s\in S$. The added elements $p_S$ 
are incomparable among themselves. The extension $D^+\colon P^+\to\ck$ is given by $D^+(p_S)=\colim_S D$. Thus, there are no new links in $P^+$ 
and $D^+$ is still $\kappa$-good. Define inductively $P^\gamma$ as $P^0=P$, $P^{\gamma+1}=(P^\gamma)^+$ and $P^\gamma=\bigcup_{\eta<\gamma}P^\eta$ 
for a limit $\gamma$. The diagrams $D^\gamma$ are defined in a similar fashion. We set $P^*=P^\kappa$ and $D^*=D^\kappa$. Since every subset of $P^*$ 
of cardinality $<\kappa$ lies in some $P^\gamma$, $\gamma<\kappa$, it has an upper bound in $P^{\gamma+1}$. Consequently, $P^*$ is $\kappa$-directed; 
it is still $\kappa$-good.
\end{proof}

\begin{theo}\label{th4.11}
Let $\ck$ be a cocomplete category and $\cx$ a class of morphisms with $\kappa$-presentable domains. Then any morphism
from $\cell(\cx)$ is a composite of a $\kappa$-good $\kappa$-directed diagram with links in $\Po(\cx)$.
\end{theo}
\begin{proof}
Let $f\in\cell(\cx)$. There is a smooth chain $(f_{\beta\alpha}\colon A_\beta\to A_\alpha)_{\beta\leq\alpha\leq\lambda}$ with links in $\Po(\cx)$ such that $f=f_{0\lambda}$.
We will proceed by recursion and prove that each $f_{0\alpha}$, $\alpha\leq\lambda$ is a composite of a $\kappa$-good $\kappa$-directed diagram 
$D_\alpha\colon P_\alpha\to\ck$ with links in $\Po(\cx)$. Moreover, $D_\beta$ is the restriction of $D_\alpha$ on the initial segment $P_\beta\subseteq P_\alpha$
for each $\beta<\alpha\leq\lambda$. We put $P_\alpha=\alpha +1$ for $\alpha <\kappa$, $P_\kappa=\kappa$ and, in both cases, $D_\alpha$
is the restriction of our chain to $P_\alpha$, $\alpha\leq\kappa$. Let $\kappa<\alpha$ and assume that the claim holds for each $\beta<\alpha$. 

Let $\alpha=\beta +1$. Then $f_{\beta\alpha}$ is a pushout
$$
\xymatrix{
A_\beta \ar[r]^{f_{\beta\alpha}} & A_\alpha \po \\
X \ar [u]^{u} \ar [r]_{g} & Y \ar[u]_{v}
}
$$
with $g$ in $\cx$. Since $X$ is $\kappa$-presentable and $D_\beta$ is $\kappa$-directed, the morphism $u\colon X\to A_\beta\cong\colim D_\beta$ factors through some $u_x\colon X\to D_\beta x$. For  $x\leq y$, we denote $u_y=D_\beta(xy)u_x$. Take pushouts
$$
\xymatrix{
D_\beta y \ar[r]^{f_{\beta y}} & A_y \po \\
X \ar [u]^{u_y} \ar [r]_{g} & Y \ar[u]_{v_y}
}
$$
By adding to the diagram $D_\beta$ the objects $A_y$, for $x\leq y$, and the obvious morphisms 
$f_{\beta y}\colon D_\beta y\to A_y$ and $A_y\to A_z$, for $x\leq y<z$, we obtain a $\kappa$-good $\kappa$-directed 
diagram $D_\alpha$. Clearly, $P_\beta\subseteq P_\alpha$ is an initial segment with a single new isolated element corresponding to $A_x$. The colimit of this diagram is $A_\alpha$. 

Let $\alpha$ be a limit ordinal and $Q$ the union of $P_\beta$, $\beta<\alpha$. Since, for $\gamma<\beta<\alpha$, $P_\gamma\subseteq P_\beta$ is an initial segment, this union $Q$ is $\kappa$-good but not necessarily $\kappa$-directed. Denoting by $E\colon Q\to\ck$ the union of $D_\beta$, $\beta<\alpha$, we define $P_\alpha=Q^\ast$ and $D_\alpha=E^*$ using the previous lemma. The links of $D_\alpha$ are those of $E$, i.e.\ those of $D_\beta$, $\beta<\alpha$, and, in particular, they lie in $\Po(\cx)$. The colimit of $D_\alpha$ is $\colim E=\colim_{\beta<\alpha}\colim D_\beta=\colim_{\beta<\alpha}A_\beta=A_\alpha$.
\end{proof}

Let $\GdDircard\kappa(\cx)$ denote the collection of all $\kappa$-good $\kappa$-directed composites with links in $\cx$. 

\begin{coro}\label{cor4.12}  
Let $\ck$ be a cocomplete category and $\cx$ a class of morphisms in $\ck_\kappa$. Then $\cell(\cx)=\GdDircard\kappa\Po(\cx)$.
\end{coro}
\begin{proof}
It follows from Proposition~\ref{prop4.4} and Theorem~\ref{th4.11}. 
\end{proof}

\begin{coro}\label{cor4.13}  
Let $\ck$ be a cocomplete category and $\cx$ a class of morphisms in $\ck_\kappa$. Then a morphism with the domain in $\ck_\kappa$
belongs to $\cell(\cx)$ if and only if it belongs to $\GdDircard\kappa\Po_\kappa(\cx)$.
\end{coro}
\begin{proof}
Let $D\colon P\to\ck$ be a $\kappa$-good diagram with a $\kappa$-presentable $D\bot$. Then all objects in the diagram $D$ are $\kappa$-presentable too: this can be seen by an easy induction on the well-founded partial ordering on $P$. The rest follows from Corollary~\ref{cor4.12}.
\end{proof}

\begin{rem}\label{re4.14}
(1) Clearly, all objects $Dx$, $x\in P$ in the $\kappa$-good diagram from the proof above are $\kappa$-presentable.

(2) The limit step in the proof of \ref{th4.11} is much simpler for $\kappa=\aleph_0$ because $Q$ is directed and thus we may take $Q^\ast=Q$.
\end{rem}

In the rest of the section we investigate cellular maps and cofibrations which are small in some respect. There are 
two possible interpretations --- either they are between $\kappa$-presentable objects or the involved transfinite composite has length $<\kappa$. We describe the relationship between these two notions of smallness.

\begin{lemma}\label{le4.15}
Let $\ck$ be a cocomplete category and $\cx$ a class of morphisms in $\ck_\kappa$ with $\kappa$ uncountable. Then 
$\cell_\kappa(\cx)=\Tcsh\kappa\Po_\kappa(\cx)$.
\end{lemma}

\begin{proof}
It is enough to show the inclusion $\cell_\kappa(\cx)\subseteq\Tcsh\kappa\Po_\kappa(\cx)$, the opposite one is easy. Thus, let $f\in\cell_\kappa(\cx)$.
Following \ref{cor4.13} and \ref{re4.14}(1), $f$ is the composite of a $\kappa$-good $\kappa$-directed diagram $D\colon P\to\ck$ with all objects
$Dx$ $\kappa$-presentable. Then the identity on $\colim D$ factors as $\colim D\to Dx_1\to\colim D$ where the second morphism 
$\delta_{x_1}\colon Dx_1\to\colim D$ is the colimit cocone component for some $x_1\in P$ and the composition $D\bot\to \colim D\to Dx_1$ of $f$ with the first morphism equals
$D(\bot x_1)$. The other composition $Dx_1\to\colim D\to Dx_1$ is idempotent 
and gets coequalized with the identity by $\delta_{x_1}$. Thus, there exists $x_2\geq x_1$ such that this pair gets coequalized already 
by $D(x_1x_2)\colon Dx_1\to Dx_2$. Proceeding inductively, we get a sequence $x_1\leq x_2\leq\cdots$ of objects of $P$ such that each $Dx_n$ is equipped 
with an idempotent that gets coequalized with the identity by $D(x_nx_{n+1})\colon Dx_n\to Dx_{n+1}$. Then, it is not hard to see that $\colim Dx_n\cong\colim D$
and thus the composite $D\bot\to\colim D$ is isomorphic to the transfinite composite of 
$$
D\bot\to Dx_1\to Dx_2\to\cdots 
$$
Since each morphism $Dx_n\to Dx_{n+1}$ lies in $\Tcsh\kappa\Po_\kappa(\cx)$ by Remark~\ref{re4.6} (and $\kappa$-presentability of $Dx_n$) and $\kappa$ 
is uncountable, the same applies to the composite.
\end{proof}

\begin{rem}\label{re4.16}
When all $\cx$-cofibrations are monomorphisms, the statement is true even for $\kappa=\omega$. This is because $Dx_1\to\colim_PD$ is then a monomorphism 
and consequently the idempotent on $Dx_1$ must be the identity, showing that the composite of the diagram is isomorphic already to $D\bot\to Dx_1$.

In general, the statement is not true for $\kappa=\omega$, as the following example shows.
\end{rem}

\begin{exam}\label{ex4.17}
Let $\kappa=\omega$ and consider the category of modules over the ring $R=\bz\oplus e\bz$ with $e^2=e$. Let $\cx=\{R\xrightarrow{e}R\}$. The transfinite 
composite of \[R\xrightarrow{e}R\xrightarrow{e}R\xrightarrow{e}\cdots\] is the map $R\to eR$ whose codomain is finitely presentable and annihilated by 
$(1-e)$. This cannot happen in $\Tcsh\omega\Po(\cx)$, since in any newly attached cell, there exists a non-zero element fixed by $(1-e)$.
\end{exam}

\begin{lemma}\label{le4.18}
Let $\ck$ be a cocomplete category and $\cx$ a class of morphisms in $\ck_\kappa$. Then $\cof_\kappa(\cx)=\Rt\Tcsh\kappa\Po_\kappa(\cx)$.
\end{lemma}

\begin{proof}
The right hand side is obviously contained in the left. For the converse, let $f\in\cof_\kappa(\cx)$ and express $f$ as a retract of some $g\in\cell(\cx)$,
\[\xymatrix{
B \ar[r] & Y \ar[r] & B \\
A \ar[r]_-\id \ar[u]_-f & A \ar[r]_-\id \ar[u]_-g & A \ar[u]_-f
}\]
(according to \ref{not2.1}(5) this retract can be taken in $A/\ck$). Express $g$ as a composite of a $\kappa$-good $\kappa$-directed diagram 
$D\colon P\to\ck$. Since $f$ is $\kappa$-presentable in $A/\ck$, it is in fact a retract of some $A\to Dx$. Following \ref{re4.14}(1), 
all objects $Dx$ are $\kappa$-presentable, which finishes the proof.
\end{proof}

The following lemma is essentially A.1.5.11 of \cite{L}. Its proof only works for locally $\kappa$-presentable categories, in contrast to our previous results. We say that a diagram $D\colon P\to\ck$ is $\kappa$-small, if $P$ has $<\kappa$ objects; its composite is then said to be a $\kappa$-small composite.

\begin{lemma}\label{le4.19}
Let $\ck$ be a locally $\kappa$-presentable category and $\cx$ a class of morphisms in $\ck_\kappa$. Then every $\kappa$-good $\kappa$-small composite 
with links in $\Po(\cx)$ lies in $\Po\cell_\kappa(\cx)$.
\end{lemma}

Later, we will also need an obvious relative version: for an initial segment $Q\subseteq P$ such that $P\smallsetminus Q$ has $<\kappa$ objects, the canonical map $\colim_QD\to\colim D$ lies in $\Po\cell_\kappa(\cx)$.

Intuitively, the lemma says that the effect of attaching $<\kappa$ cells to an object takes place in some $\kappa$-presentable part. Attaching the cells 
solely to this small part results in a cellular map between $\kappa$-presentable objects with the original map being its pushout.

\begin{proof}
Let $D\colon P\to\ck$ be a $\kappa$-good $\kappa$-small diagram with links in $\Po(\cx)$. Express the bottom object of the composite $D\bot\to\colim D$ as a $\kappa$-filtered colimit $D\bot=\colim_{i\in\ci}A_i$ of a diagram $A\colon\ci\to\ck$ of $\kappa$-presentable objects such that $\ci$ has $\kappa$-small colimits and $A$ preserves them. For
instance, we can take the canonical diagram $\ci=\ck_\kappa/D\bot$ and its projection $A$ sending $X\to D\bot$ to $X$.

We will construct inductively a smooth chain $i_Q\in\ci$, indexed by initial segments $Q$ of $(P,\preceq)$ as in Proposition~\ref{prop4.5}, whose images under $A$ are denoted $A_Q=Ai_Q$, and morphisms $f_Q\colon A_Q\to B_Q$ in $\cell_\kappa(\cx)$ such that $D\bot\to\colim_QD$ is a pushout of $f_Q$ along the component $A_Q\to D\bot$ of the colimit cocone,
\[\xymatrix{
D\bot \ar[r] & \colim_QD \po \\
A_Q \ar[r]_-{f_Q} \ar[u] & B_Q \ar[u]
}\]
These data are subject to the following two conditions:
\begin{enumerate}
\item
for a successor $Q'\subseteq Q$, the morphism $f_Q$ is a composition of the pushout of $f_{Q'}$ along the obvious morphism $A_{Q'}\to A_Q$ with some element of $\Po(\cx)$;
{\addtocounter{equation}{-1}\renewcommand{\theequation}{$\star$}
\begin{equation}\label{eq:ind_step}\xymatrix{
A_{Q} \ar[r] \ar@/^15pt/[rr]^-{f_Q} & A_Q\sqcup_{A_{Q'}}B_{Q'} \ar[r] & B_Q \\
{}\POS[]+<-15pt,0pt>*+!!<0pt,\the\fontdimen22\textfont2>{A_{Q'}}="x" \ar[u] &
{}\POS[]+<-15pt,0pt>*+!!<0pt,\the\fontdimen22\textfont2>{B_{Q'}} \ar@{<-}"x"^-{f_{Q'}} \ar[u]
\POS[u]+/-\stdpbsize/="p",\xycorner{[u]+<-\stdpbsize,0pt>}{[u]+<-\stdpbsize,0pt>+/-\stdpbsize/}{"p"}
\POS[]+<15pt,0pt>*+!!<0pt,\the\fontdimen22\textfont2>{X}="y" \ar[u] &
{}\POS[]+<15pt,0pt>*+!!<0pt,\the\fontdimen22\textfont2>{Y} \ar@{<-}"y"^-{g_Q} \ar[u]
\POS[u]+/-\stdpbsize/="p",\xycorner{[u]+<-\stdpbsize,0pt>}{[u]+<-\stdpbsize,0pt>+/-\stdpbsize/}{"p"}
}\end{equation}}
\item
for a limit $Q$, the morphism $A_Q\to B_Q$ is a colimit of the pushouts of $A_{Q'}\to B_{Q'}$ over $Q'\subseteq Q$. (In particular, it is a transfinite composite of pushouts of the $g_{Q'}$ with $Q'\subsetneq Q$.)
\end{enumerate}
For $Q=P$ we obtain $D\bot\to\colim D$ as a pushout of $A_P\to B_P$ that lies in $\cell_\kappa(\cx)$ as desired.

Since the limit steps are determined by condition (2), it remains to describe the successor case. Let $Q'\subseteq Q$ be successive initial segments and assume that the only element $x$ of $Q'\smallsetminus Q$ is isolated --- otherwise, we may take $f_Q=f_{Q'}$. By induction, the partial composite $D\bot\to\colim_{Q'}D$ is a pushout of $f_{Q'}\colon A_{Q'}\to B_{Q'}$ lying in $\cell_\kappa(\cx)$. Then $\colim_{Q'}D$ is a colimit of the $\kappa$-filtered diagram of pushouts $Aj\sqcup_{A_{Q'}}B_{Q'}$, indexed by arrows $i_{Q'}\to j$. The morphism $\colim_{Q'}D\to\colim_QD$ is a pushout of $Dx^-\to Dx$ and thus a pushout of some $X\to Y$ in $\cx$. The attaching map $X\to Dx^-\to\colim_{Q'}D$ factors through some $A_j\sqcup_{A_{Q'}}B_{Q'}$. We set $i_Q=j$ and obtain $f_Q$ as in \eqref{eq:ind_step}.
\end{proof}

\begin{coro}\label {cor4.20}
Let $\ck$ be a locally $\kappa$-presentable category and $\cx$ a class of morphisms in $\ck_\kappa$. Then
\[\Tcsh\kappa\Po(\cx)=\Po\Tcsh\kappa\Po_\kappa(\cx).\]
\end{coro}

\begin{proof}
The right hand side is clearly contained in the left. The opposite inclusion is easily implied by the previous lemma.
\end{proof}

\section{Applications}
An object $K$ of a category $\ck$ is called $\cx$-\textit{cofibrant} if the unique morphism $0\to K$ from an initial object belongs to $\cof(\cx)$;
$\cx$-\textit{cellular} objects are defined analogously.

\begin{coro}\label{cor5.1}
Let $\ck$ be a cocomplete category and $\cx$ a class of morphisms in $\ck_\kappa$. Then any $\cx$-cofibrant object of $\ck$ is 
a $\kappa$-filtered colimit of $\kappa$-presentable $\cx$-cofibrant objects. 
\end{coro}
\begin{proof}
For cellular objects the claim follows from \ref{cor4.13} and \ref{re4.14}(1). Since any cofibrant object is a retract of a cellular one, the result follows from
\cite{MP} 2.3.11. (the proof applies in the case when $\ck$ is not locally $\kappa$-presentable but merely cocomplete). 
\end{proof}

\begin{coro}\label{cor5.2}
Let $\kappa$ be an uncountable regular cardinal, $\ck$ a locally $\kappa$-presentable category and $\cx$ a class of morphisms in $\ck_\kappa$. 
Then any $\cx$-cofibrant object of $\ck$ is a $\kappa$-good $\kappa$-directed colimit of $\kappa$-presentable $\cx$-cofibrant objects
where all links are $\cx$-cofibrations. 
\end{coro}

\begin{proof}
The proof is the same as in Corollary~\ref{cor5.1} but we use Theorem~\ref{t:retract_removal} instead of \cite{MP} 2.3.11.
\end{proof}

\begin{rem}\label{re5.3}
(1) According to the proof of Corollary~\ref{cor5.2}, the links even lie in $\Po\cof_\kappa(\cx)$.

(2) Let $\ck$ be the category of modules over a ring $R$ and let $\cp$ be the class of projective $R$-modules. A monomorphism $f\colon A\to B$ is called
a $\cp$-monomorphism if its cokernel is a projective module $P$. Then $f$ is the coproduct injection $A\to A\oplus P$. We get a weak factorization 
system $(\cp$-$\Mono,\Epi)$ whose left class consists of all $\cp$-monomorphisms and the right class of all epimorphisms. The left class is cofibrantly 
generated by a morphism $i\colon 0\to R$. Cofibrant objects are precisely projective modules. Following \ref{cor5.2}, every projective module is 
a $\omega_1$-good $\omega_1$-directed colimit of $\omega_1$-presentable projective modules where all links are $\cp$-monomorphisms. Hence all
morphisms of the corresponding diagram are coproduct injections. Thus every projective module  is a coproduct of countably generated projective modules,
which is a classic theorem due to Kaplansky. 

This also shows that Corollary~\ref{cor5.2} cannot be extended to $\omega$ because there exist rings which admit projective modules which are not coproducts
of finitely generated projective modules (see \cite{GT} 7.15).
\end{rem}

Let $\ck$ be a Grothendieck category. Given a class $\cs$ of objects of $\ck$, a monomorphism $f$ is called an $\cs$-monomorphism if its cokernel 
belongs to $\cs$. An object $K$ is $\cs$-filtered if the unique morphism $0\to K$ is a transfinite composite of $\cs$-monomorphisms.  A class $\cc$ 
is \textit{deconstructible} if it is the class of $\cs$-filtered objects for a set $\cs$ (see \cite{S}). 

\begin{rem}\label{re5.4}
A class $\cc$ is deconstructible if and only if $\cc$-mo\-no\-mor\-phisms are the cellular closure of a set of morphisms. 
 
Sufficiency is easy because if $\cc$-monomorphisms are $\cell(\cx)$ for a set $\cx$ then the set $\cs$ of cokernels of morphisms from $\cx$ makes $\cc$ 
deconstructible. Necessity is \cite{SS}, Proposition 2.7.
\end{rem}
 
\begin{rem}\label{re5.5}
Let $R$ be a ring and $(\ca,\cb)$ a cotorsion pair of finite type, i.e., generated by a set $\cs$ of finitely presentable $R$-modules.
Any $A\in\cs$ is a quotient $p_A\colon A^\ast\to A$ of a free module and $\ker(p_A)$ is a morphism between finitely presentable modules. Then 
$\ca$-monomorphisms are cellularly generated by $\ker(p_A)$, $A\in\cs$ (see \cite{SS} as above). Following Corollary~\ref{cor5.1}, any module
from $\ca$ is a directed colimit of finitely presentable modules from $\ca$. This fact was proved in \cite{AHT} 2.3.
\end{rem}

An object $K$ of a model category $\ck$ is called \textit{acyclic} if $K\to 1$ is a weak equivalence.

\begin{coro}\label{cor5.6}
Let $\ck$ be a $\kappa$-combinatorial model category where $1$ is $\kappa$-presentable and any morphism $K\to 1$ splits by a cofibration. 
Then any acyclic object of $\ck$ is a $\kappa$-directed colimit of $\kappa$-presentable acyclic objects. 
\end{coro}
\begin{proof}
Let $\ck_\ast=1\downarrow\ck$ be the associated pointed model category (see \cite{Ho} 1.1.8). Let $K$ be an acyclic object of $\ck$.
Following our assumption, there is a cofibration $f\colon 1\to K$ and, since $K$ is acyclic, $f$ is a trivial cofibration. Thus $(K,f)$
is trivially cofibrant in $\ck_\ast$ (see \cite{Ho} 1.1.8). Since $\ck_\ast$ is $\kappa$-combinatorial as well, \ref{cor5.1} applied to trivial
cofibrations in $\ck_\ast$ yields that $(K,f)$ is a $\kappa$-directed colimit of trivially cofibrant $\kappa$-presentable objects $(K_i,f_i)$. 
Since $1$ is $\kappa$-presentable, any $K_i$ is $\kappa$-presentable in $\ck$. Since each $f_i$ is a trivial cofibration in $\ck$, each $K_i$ 
is acyclic in $\ck$.
\end{proof}

\begin{rem}\label{re5.7}
In particular, any acyclic simplicial set is a directed colimit of finitely presentable acyclic simplicial sets (see \cite{JW} 6.3).
\end{rem}

The authors are grateful to Jan \v{S}t\kern-.23em\lower-.2ex\hbox{'}ov\'{i}\v{c}ek for useful discussions.

\appendix
\section{General fat small object argument}

In this section, let $\kappa$ be an arbitrary regular cardinal and let $\cx$ be a set of morphisms in $\ck_\kappa$. We will describe in this appendix 
an alternative to the usual small object argument.

\begin{theo}
Let $f\colon A\to B$ be a morphism. Then, there exists a factorization of $f$, whose left part lies in $\GdDircard\kappa\Po(\cx)$ and whose right part 
lies in $\cx^\square$.
\end{theo}

\begin{proof}
The factorization is obtained as a transfinite iteration of length $\kappa$ of the construction, that (similarly to the usual small object argument) 
adds cells that solve all the possible lifting problems. Thus, we consider by induction, for $\alpha<\kappa$, a factorization \[A\to\colim D_\alpha\to B\]
where $D_\alpha\colon P_\alpha\to K$ is a $\kappa$-good $\kappa$-directed diagram with links in $\Po(\cx)$. We assume, that for each $\beta<\alpha$, 
$P_\beta$ is an initial segment in $P_\alpha$ and that $D_\beta$ is the restriction of $D_\alpha$. Consider the set of all squares
\[\xymatrix{
\colim D_\alpha \ar[r]^-f & B \\
X \ar[r]_-g \ar[u]^-x & Y \ar[u]_-y
}\]
parametrized by $x,y$, and for each such square, choose a factorization of $x$ through some $X\to D_\beta$. Then form the pushout square
\[\xymatrix{
D_\beta \ar[r] & D_{x,y} \po \\
X \ar[r] \ar[u] & Y \ar[u]
}\]
Next, add to $P_\alpha$, for each $x,y$, objects $p_{x,y}$ with $p_{x,y}>\beta$ and extend the diagram $D_\alpha$ to $p_{x,y}$ by $D_{x,y}$. Finally, 
to obtain $D_{\alpha+1}\colon P_{\alpha+1}\to\ck$, perform the ${}^*$-construction of Lemma~\ref{le4.10}. In the limit steps, take $P_\alpha$ to be 
the ${}^*$-construction of the union $\bigcup_{\beta<\alpha}P_\beta$.

Thus, in each step $\alpha\leq\kappa$, we obtain a $\kappa$-good $\kappa$-directed diagram $D_\alpha\colon P_\alpha\to\ck$ with links in $\Po(\cx)$. 
The factorization is \[A\to\colim D_\kappa\to B.\] It is easy to see that $\colim D_\kappa\to B$ lies in $\cx^\square$ --- since any 
$X\to\colim D_\kappa$ factors through some $\colim D_\alpha$ with $\alpha<\kappa$, the lifting problem is solved in $\colim D_{\alpha+1}$.
\end{proof}

\begin{rem}
There is a slight difference between the proof of this theorem and a direct application of Theorem~\ref{th4.11} to the usual small object argument: 
here we attach a number of cells at once and only then apply the ${}^*$-construction. We could have developed the rearrangement in Theorem~\ref{th4.11} 
for transfinite composites of pushouts of \emph{coproducts} of morphisms in $\cx$ which would have given this exact version of the small object argument.
\end{rem}

\section{Elimination of retracts}
Retracts of $A\to B$ in this section are to be understood as retracts in the category $A/\ck$. These are enough to produce all cofibrations 
as retracts of cellular morphisms, as explained in \ref{not2.1}(5). Moreover we assume that $\ck$ is locally $\kappa$-presentable and that $\cx$ 
is a set of morphisms in $\ck_\kappa$ for some \emph{uncountable} cardinal $\kappa$.

The following theorem shows, that the use of retracts is unnecessary, at least if one is willing to enlarge the generating set $\cx$. This result 
was proved by Lurie in \cite{L}, A.1.5.12. Here we present an alternative proof, that relies on $\kappa$-good $\kappa$-directed diagrams.

\begin{theo}\label{t:retract_removal}
$\cof(\cx)=\cell(\cof_\kappa(\cx))$.
\end{theo}

We start with a couple of generalities.

Let $P$ be a $\kappa$-good $\kappa$-directed poset. We say, that an upper bound $x$ of an initial segment $Q\subseteq P$ is a \emph{strong upper bound}, 
if all the elements in $(\downarrow x)\smallsetminus Q$ are limit. Equivalently\footnote{If $x$ is a strong upper bound of $Q$ then the inclusion 
$Q\subseteq(\downarrow x)$ induces an isomorphism $\colim_QD\to\colim_{\downarrow x}D\cong Dx$ by Lemma~\ref{le4.9}. In the opposite direction, 
we observe first that a representable functor $D=P(y,-)\colon P\to\Set$ is smooth if and only if $y$ is isolated. Thus, if an isolated 
$y\in(\downarrow x)\smallsetminus Q$ existed, we would then get $D|_Q=\emptyset$ and $Dx=1$, a contradiction.}, the canonical map $\colim_QD\to Dx$ 
is an isomorphism for all smooth diagrams $D\colon P\to\ck$. We define
\[\overline Q=\{x\in P\mid \text{$x$ is a strong upper bound of some subset $R\subseteq Q$}\}.\]
By its definition, the closure $\overline Q$ is an initial segment and, according to Lemma~\ref{le4.9}, the canonical map 
$\colim_{Q}D\to\colim_{\overline Q}D$ is an isomorphism. We say that an initial segment $Q$ is \emph{closed}, if $\overline Q=Q$.

It is obvious from its construction in Lemma~\ref{le4.10} that $P^*$ has strong upper bounds of all $\kappa$-small initial segments.

\begin{lemma}\label{l:lifting_idempotents}
Let $P$ be a $\kappa$-good poset with strong upper bounds of all $\kappa$-small initial segments. Let $D\colon P\to\ck$ be a smooth diagram whose 
all objects are $\kappa$-presentable and let there be given an idempotent $f\colon\colim D\to\colim D$ in $(D\bot)/\ck$.

Then there exists an endofunctor $S\colon P\to P$ with $S\bot=\bot$ and $x\leq Sx$, and an idempotent natural transformation $\varphi\colon DS\to DS$ 
with $\varphi_\bot=\id$, such that the following diagram commutes.
\[\xymatrix{
\colim D \ar[r]^-f \ar[d]_-\cong & \colim D \ar[d]^-\cong\\
\colim DS \ar[r]_-{\varphi_*} & \colim DS
}\]
(The vertical maps are induced by $D(x,Sx)\colon Dx\to DSx$.)

Moreover, if $Q\subseteq P$ is a closed initial segment and the idempotent $f$ extends an idempotent $f'\colon\colim_QD\to\colim_QD$, and if there 
are given the $S'\colon Q\to Q$ and $\varphi'\colon DS'\to DS'$ as above for $f'$, then the $S$ and $\varphi$ may be constructed as extensions 
of $S'$ and $\varphi'$.
\end{lemma}

This lemma roughly says that there are many objects in the diagram with idempotents on them (they are cofinal in $P$). If these constituted a good diagram, we could have used them to express the image of the idempotent $f$ as a good colimit of retracts. This is however not true in general and that is why we need the relative version.

\begin{proof}
The basic idea is rather simple. We construct $Sx$ as a strong upper bound of a chain $S_1x\leq S_2x\leq\cdots$ and $\varphi_x$ as a colimit of morphisms $(\varphi_n)_x$ in the diagram
\[\xymatrix{
DS_1x \ar[r] \ar[dr]^(.6){(\varphi_1)_x} & DS_2x \ar[r] \ar[dr]^(.6){(\varphi_2)_x} & \cdots \ar[r] & DSx \ar[d]^-{\varphi_x} \ar[r] & \colim_PD \ar[d]^-f \\
DS_1x \ar[r] & DS_2x \ar[r] & \cdots \ar[r] & DSx \ar[r] & \colim_PD
}\]
where the unnamed arrows are induced by $D$ from the unique arrows in $P$. We will stick to this convention in the rest of the proof.

The $S_nx$ and $(\varphi_n)_x$ are constructed inductively. Without the requirement of naturality, they are obtained by factoring
\[DS_nx\to\colim D\xrightarrow{f}\colim D\]
as $DS_nx\xrightarrow{(\varphi_n)_x}DS_{n+1}x\to\colim D$. By choosing $S_{n+1}x$ big enough, we may assume that $S_{n+1}x\geq S_nx$ and that the following compositions are equal
{\addtocounter{equation}{-1}\renewcommand{\theequation}{$\Diamond$}\begin{equation}\label{e:idempotency_phi}
\xy *!C\xybox{
\xymatrix@C=30pt{
& DS_nx \ar[rd] \\
DS_{n-1}x \ar[r]^-{(\varphi_{n-1})_x} \ar[ru]^-{(\varphi_{n-1})_x} \ar[rd] & DS_nx \ar[r]^-{(\varphi_n)_x} & DS_{n+1}x \\
& DS_nx \ar[ru]_-{(\varphi_n)_x}
}}\endxy\end{equation}}
ensuring that $\varphi_x$ will be idempotent.

To ensure naturality, we have to construct $(\varphi_n)_x$ inductively with respect to $x$. We set $S_{n+1}\bot=\bot$ and $(\varphi_n)_\bot=\id$. Assume, that $(\varphi_n)_y$ has been defined for all $y<x$. We thus have a diagram
\[\xymatrix{
\colim D \ar[rr]^-f & & \colim D\\
DS_nx \ar[u] \ar@{-->}[r]^-{(\varphi_n)_x} & DS_{n+1}x \ar@{-->}[ru] \\
\colim_{y<x}DS_ny \ar[rr]_-g \ar[u] & & Dz \ar[uu] \ar@{-->}[lu]
}\]
where $z$ is an arbitrary upper bound of the set $\{S_{n+1}y\mid y<x\}$ and the bottom map $g$ is given on $DS_ny$ as the composition
\[DS_ny\xrightarrow{(\varphi_n)_y}DS_{n+1}y\to Dz.\]
The solid square commutes and it is easy to find some $S_{n+1}x$ and a factorization $(\varphi_n)_x$ using the $\kappa$-presentability of $DS_nx$ 
and $\colim_{y<x}DS_ny$. In this way both $S_{n+1}$ will be a functor and $\varphi_n$ a natural transformation. We may still achieve the commutativity of \eqref{e:idempotency_phi} by passing to a bigger $S_{n+1}x$.

Similarly, the strong upper bound $Sx$ of the chain $S_1x\leq S_2x\leq\cdots$ is chosen inductively, starting with $S\bot=\bot$. When all the $Sy$ have been chosen for $y<x$, $Sx$ is chosen as a strong upper bound for the initial segment spanned by $S_1x,S_2x,\ldots$ and all the $Sy$ with $y<x$. At the same time $Sx$ is a strong upper bound of the initial segment spanned by the $S_1x,S_2x,\ldots$ since all the $Sy$, $y<x$, lie in the closure of this initial segment, so that Lemma~\ref{le4.9} applies.

When $S'$ and $\varphi'$ are given, we may choose $S_{n+1}x=S'x$, $(\varphi_n)_x=\varphi'_x$, and $Sx=S'x$ in the above, whenever $x\in Q$.
\end{proof}

\begin{proof}[Proof of Theorem~\ref{t:retract_removal}]
Suppose that $A\to B$ is a cellular map and that a retract of it is given by an idempotent $f\colon B\to B$ in $A/\ck$. We write $A\to B$ as a colimit of a $\kappa$-good $\kappa$-directed diagram $D\colon P\to A/\ck$ with links in $\Po(\cx)$. By our assumptions, it consists of $\kappa$-presentable objects of $A/\ck$ and, applying the ${}^*$-construction of Lemma~\ref{le4.10} if necessary, we may construct $D$ in such a way that strong upper bounds of all $\kappa$-small initial segments exist. Thus, Lemma~\ref{l:lifting_idempotents} is applicable to $D$.

We may construct the colimit $\colim D$ inductively similarly to the proof of Proposition~\ref{prop4.5}. It will be a transfinite composite of partial colimits $\colim_{P_i}D$ equipped with compatible idempotents
\[f_i\colon\colim_{P_i}D\to\colim_{P_i}D,\]
where the $P_i$ form a transfinite sequence of closed initial segments with respect to the inclusion.

We start with $P_0=\{\bot\}$ and $f_0=\id_{D\bot}$.

Suppose, that we have constructed $P_i$ and $f_i$. Then, we construct a new endofunctor $S$ and a natural transformation $\varphi$ on $P$ by Lemma~\ref{l:lifting_idempotents} by first constructing them on $P_i$ and then extending to $P$. Next, take any minimal element $x$ not in $P_i$ and denote by $Q$ the initial segment generated by the sequence $Sx,S^2x,\ldots$; it is $\kappa$-small. Then the colimit of $D$ over $P_i\cup Q$ can be written as the pushout
\[\xymatrix{
\colim_{P_i}D \ar[r] & \leftbox{\colim_{P_i\cup Q}D}{{}\xrightarrow\cong\colim_{P_{i+1}}D} \po \\
\colim_{P_i\cap Q}D \ar[r] \ar[u] & \colim_QD \ar[u]
}\]
Finally we take $P_{i+1}=\overline{P_i\cup Q}$. This does not change the colimit. The functor $S$ preserves both $P_i$ and $Q$ by construction and thus also $P_i\cap Q$ and $P_i\cup Q$. Therefore, we have idempotents on all colimits in the above square and they are compatible; we denote that on $\colim_{P_{i+1}}D$ by $f_{i+1}$.

Finally, we have to explain how to compute the retract of the composite $D\bot\to\colim_PD$. We have split this composite into a transfinite composite in the category of objects with idempotents. For each $i$, let $E_i$ denote the image of the idempotent on $\colim_{P_i}D$ with the retraction $r_i$. Then consider the following pushout (which simply defines $F_i$)
\[\xymatrix{
E_i \ar[r] & F_i \po \\
\colim_{P_i}D \ar[r] \ar[u]^-{r_i} & \colim_{P_{i+1}}D \ar[u]
}\]
There is an induced idempotent on $F_i$, which restricts to $\id$ on $E_i$ and whose image is exactly $E_{i+1}$. At the same time, $E_i\to F_i$ is a pushout of the map $\colim_{P_i\cap Q}D\to\colim_QD$ which, by $\kappa$-smallness of $Q$ and a relative version of Lemma~\ref{le4.19}, is a pushout of an element of $\cell_\kappa(\cx)$. By Lemma~\ref{l:short_retracts} below, $E_i\to E_{i+1}$ lies in $\Po\cof_\kappa(\cx)$. Thus, the composite $A\to\colim_iE_i$, i.e.\ the image of the idempotent $f$ that we started with, is $\cof_\kappa(\cx)$-cellular.
\end{proof}

The following lemma is A.1.5.10 of \cite{L}.

\begin{lemma}\label{l:short_retracts}
$\Rt\Po\cell_\kappa(\cx)\subseteq\Po\cof_\kappa(\cx)$.
\end{lemma}

\begin{proof}
Let $X\to Y$ be a pushout of a morphism $A\to B$ from $\ck_\kappa$ and let $f\colon Y\to Y$ be an idempotent in $X/\ck$. We want to express its image as an element of $\Po\cof_\kappa(\cx)$. Write $X$ as the canonical $\kappa$-filtered colimit $X=\colim X_\alpha$ of $\kappa$-presentable objects of $A/\ck$. Corresponding to this, $Y$ is a colimit $Y=\colim Y_\alpha$, where $Y_\alpha=X_\alpha\sqcup_AB$. As this diagram has all $\kappa$-small colimits, we may use the proof of Lemma~\ref{l:lifting_idempotents} to find a chain $\alpha_1\to\alpha_2\to\cdots$ together with morphisms $\varphi_n\colon Y_{\alpha_n}\to Y_{\alpha_{n+1}}$ that induce an idempotent $f_\alpha$ on $Y_\alpha=\colim_nY_{\alpha_n}$. Since $f$ restricts to $\id$ on $X$, we may assume at each point, that $\varphi_n$ restricts to the map $X_{\alpha_n}\to X_{\alpha_{n+1}}$ in the canonical diagram (by passing to ``bigger'' $\alpha_{n+1}$ if necessary). In this way, $f_\alpha$ will be an idempotent in $X_\alpha/\ck$. Then the image of $f$ is a pushout of the image of $f_\alpha$, as required.
\end{proof}

\end{document}